\documentclass[reqno,11pt]{amsart}
\usepackage{amsmath, latexsym, amsfonts, amssymb}
\usepackage{graphics,epsf,psfrag}

\usepackage{stmaryrd}

\numberwithin{equation}{section}

   \newtheorem{theorem}{Theorem}[section]
   
    \newtheorem{lemma}[theorem]{Lemma}
   \newtheorem{proposition}[theorem]{Proposition}
   \newtheorem{rem}[theorem]{Remark}
     
   \newtheorem{conjecture}[theorem]{Conjecture}

\renewcommand{\tilde}{\widetilde}
\renewcommand{\hat}{\widehat}
\newcommand{\tf}{\textsc{f}}
\newcommand{\cc}{\complement}
\newcommand{\cf}{\mathrm{f}}
\newcommand{\lint}{\llbracket}
\newcommand{\rint}{\rrbracket}

\newcommand{\cA}{\ensuremath{\mathcal A}}

\newcommand{\cN}{\ensuremath{\mathcal N}}

\newcommand{\cT}{\ensuremath{\mathcal T}}

\newcommand{\bbE}{{\ensuremath{\mathbb E}} }

\newcommand{\bbN}{{\ensuremath{\mathbb N}} } 
 
\newcommand{\bbP}{{\ensuremath{\mathbb P}} } 
 
\newcommand{\bbR}{{\ensuremath{\mathbb R}} }

\newcommand{\bbZ}{{\ensuremath{\mathbb Z}} } 

\newcommand{\bt}{{\mathbf t}}

\newcommand{\ga}{\alpha}
\newcommand{\gb}{\beta}
\newcommand{\gga}{\gamma}            
\newcommand{\gd}{\delta}
\newcommand{\gep}{\varepsilon}       

\newcommand{\gG}{\Gamma}

\newcommand{\go}{\omega}

\newcommand{\bP}{{\ensuremath{\mathbf P}} }
\newcommand{\bE}{{\ensuremath{\mathbf E}} }

\newcommand{\bbbP}{\mathbb{P}}
\newcommand{\E}{\mathbb{E}}

\newcommand{\N}{\mathbb{N}}
\newcommand{\ind}{\mathbf{1}}

\newcommand{\tZ}{\tilde{Z}}

\DeclareMathSymbol{\leqslant}{\mathalpha}{AMSa}{"36} 
\DeclareMathSymbol{\geqslant}{\mathalpha}{AMSa}{"3E} 
\DeclareMathSymbol{\eset}{\mathalpha}{AMSb}{"3F}     
\renewcommand{\leq}{\;\leqslant\;}                   
\renewcommand{\geq}{\;\geqslant\;}                   
\newcommand{\dd}{\,\text{\rm d}}             
\newcommand{\sumtwo}[2]{\sum_{\substack{#1 \\ #2}}} 

\title[$\gamma$-stable pinning model]{Disorder relevance without Harris Criterion: \\
the case of pinning model with $\gamma$-stable environment}

\author{Hubert Lacoin}
\address{
  IMPA, Institudo de Matem\'atica Pura e Aplicada, Estrada Dona Castorina 110
Rio de Janeiro, CEP-22460-320, Brasil. 
}

\author{Julien Sohier}
\address{
Laboratoire d'Analyse et de Math\'ematiques Appliqu\'ees (LAMA), UMR 8050
Universit\'e Paris Est Cr\'eteil
61, avenue du G\'en\'eral de Gaulle, 94010 Cr\'eteil Cedex, France
}

\begin{document}

\begin{abstract}
We investigate disorder relevance for the pinning of a renewal whose inter-arrival law has tail exponent $\alpha>0$ 
when the law of the random environment is in the domain of attraction of a stable law with parameter $\gamma \in (1,2)$.
We prove that in this case, the effect of disorder is not decided by the sign of the specific heat exponent as predicted by Harris criterion but that a new criterion emerges to decide 
disorder relevance.
More precisely we show that when $\alpha>1-\gamma^{-1}$ there is a shift of the critical point at every temperature 
whereas when $\alpha< 1-\gamma^{-1}$, at high temperature the quenched and annealed critical point coincide, and the critical exponents are identical. 
\\[10pt]
\\[10pt]
\end{abstract}
\maketitle

\section{Introduction}

 It is a common practice in Theoretical Physics to use simplified lattice models in order to study the qualitative behavior of systems with a large number of interacting components. 
A prototypical example 
   is the usual Ising model on $\bbZ^{d}$, which has been used to understand the phenomenon of ferromagnetic transition.
   The reason why these lattice models are believed to yield a fair approximation of real world phenomena is that critical phenomena such as phase transitions 
are not supposed to rely on the detailed structure of interaction and should be preserved after replacing the original system by its simpler lattice version.
 In their simplest expression
   these models display homogeneous interactions in the sense that the Hamiltonian is invariant by the lattice symmetries.

\medskip

 Since real world materials always present at least some infinitesimal lack of regularity, an important issue to assert the validity of this approach
is thus whether the qualitative behavior of a lattice system
(such as the presence of a phase transition, its order, etc...) remains the same when a small amount of irregularity is introduced. This natural 
 question leads to consider "disordered" versions of the models where the interactions are given by the realization of an ergodic field; the simplest example being 
  the case of a field of independent identically distributed random variables.

\medskip

An important step in the understanding of the influence of disorder is due to Harris \cite{cf:Harris} which introduced a celebrated criterion which allows 
to predict whether a small amount of disorder may or may not change the critical properties of the system. Harris criterion relies on the analysis of the 
specific-heat exponent near the critical point and is backed by a heuristic which relies on computing the second moment of the partition function.

\medskip

The validity of this criterion has been checked in a few cases where the homogeneous model is well understood, in particular for the random field Ising model
\cite{cf:IM, cf:AW}.
A case which has generated a rich literature in the past decade is that of pinning of a one dimensional polymer on a defect line:
a model where a renewal process with power-law tails has its law modified due to energetic interactions at renewal points
\cite{cf:DHV, cf:GT05, cf:DGLT07, cf:AZ08, cf:CdH, cf:A06, cf:T08, cf:Lmart, cf:BL}.
We refer to the monographs \cite{cf:GB, cf:GB2} for a detailed introduction to the subject.

\medskip

For this particular family of models, proofs of  disorder relevance (and by this we mean: a radical 
change in the critical properties of the system) \cite{cf:GT05, cf:DGLT07, cf:AZ08, cf:CdH} 
or irrelevance \cite{cf:A06, cf:T08, cf:Lmart} 
have been given, depending on the value 
of the exponent associated to the renewal times, all of which confirmed the validity of Harris criterion.

\medskip

The heuristics behind Harris criterion relies on controlling the asymptotic behavior of the variance 
of the partition function of the system at the critical point of the system without disorder. As a consequence, a natural question occuring is the following:  
 how could this criterion be altered if one considers systems with heavy tail environments, for which
the second moment of the partition function is infinite. We stress that this question is far from being a purely technical one; 
indeed, heavy tail of the environment is likely to create greater fluctuations of the thermodynamics quantities around their
average value, and this could in principle alter the validity of the criterion.

\medskip

 In the particular case of the pinning model detailed heuristic second moment computations were performed in \cite{cf:DHV} to predict whether disorder is relevant, and 
 in the relevant case in which way  the critical point is shifted.
These results have since been made rigorous, and in particular the papers which prove lower-bound results
on the free energy \cite{cf:A06, cf:T08, cf:Lmart} all rely on controlling the second moment of a partition function.

\medskip

In the present paper, we study the disordered pinning model in the case where the environment is IID but with a distribution belonging to the domain of attraction of a 
stable law with parameter $\gamma \in (1,2)$; this entails in particular that the disorder has a first moment but an infinite second moment. 
Our aim is to show that in that case, the model falls into a different universality class and that the original 
formulation of the Harris criterion is not valid anymore. We present and prove an alternative formulation of the criterion in that case, 
which we believe should hold for a wide class of disordered systems.

 \section{Model and results}

   \subsection{Setup}

  Let $\tau = (\tau_{n})_{n \geq 0}$ be a recurrent integer valued renewal process, 
  that is a random sequence whose increments $(\tau_{n+1}-\tau_{n})$ are independent 
   identically distributed (IID) positive integers. We denote by $\bP$ the associated probability distribution. 
   We assume that $\tau_{0} = 0$, and that the inter-arrival distribution is power-law or more precisely that it satisfies  
  \begin{equation}\label{Kt}
   K(n) :=  \bP[\tau_{1} = n] \stackrel{n\to \infty}{\sim} C_K n^{-(1+\ga)}, \ga \in (0,\infty),
  \end{equation} 
where $C_K>0$ is an arbitrary constant. 
Note that $\tau$ can alternatively be considered as an infinite
   subset of $\bbN$ and in our notation, $\{ n\in \tau \}$ is equivalent to $\{ \exists k\in \bbN, \ \tau_k=n\}$.

 \medskip

We consider a sequence of IID random variables $(\go_{n})_{n\ge 0}$ and denote its law by $\bbP$. 
For our main results to hold, we will make some specific assumptions on the distribution of the $\go$'s \eqref{defgamma}, but 
for the sake of this introduction we will only assume that
\begin{equation}\label{star}
\bbP[\go_1\ge -1]=1 \quad \text{ and } \E[\go_1]=0.
\end{equation}
    Given $\gb  \in [0,1)$, $h \in \bbR$, and $N  \in \N$, we define 
    a modified renewal measure $\bP^{\gb,\go}_{N,h}$ whose 
    Radon-Nikodym derivative with respect to $\bP$ is given by 
   \begin{equation}\label{defmod}
    \frac{\dd\bP^{\gb,\go}_{N,h}}{\dd \bP}(\tau) = \frac{1}{ Z^{\gb,\go}_{N,h}} \left( \prod_{n\in \tau \cap [1,N]} e^h(\gb\go_n+1) \right) \delta_N   
   \end{equation} 
    where $\gd_{n} := \ind_{\{n \in \tau\}}$ and 
    \begin{equation}\label{partfunc}
     Z^{\gb,\go}_{N,h} = \bE\left[ \left( \prod_{n\in \tau \cap [1,N]}  \!\!\!\!\! e^h(\gb\go_n+1) \right) \delta_N \right] = \bE\left[\prod_{n=1}^N  \left(1+ [e^h(\gb\go_n+1)-1]\gd_{n}\right)\delta_N \right]
    \end{equation} 
    is the partition function.

    \noindent In the case $\gb=0$, we retrieve the homogeneous pinning model where 
    \begin{equation}
        \frac{\dd\bP_{N,h}}{\dd \bP}(\tau):=\frac{1}{Z_{N,h}}e^{h\sum_{n=1}^N \delta_n}\delta_N  
        \quad \text{and} \quad Z_{N,h}:=\bE\left[ e^{h\sum_{n=1}^N \delta_n}\delta_N \right].
    \end{equation}
    Note that we assumed that the renewal is recurrent, that is 
    $$\bP[\tau_{1} = \infty] :=1-\sum_{n=1}^{\infty} K(n)=0.$$ 
     It is a classic observation (see the e.g.\  \cite[Remark 1.19]{cf:GB}) that this yields no loss of generality: 
 in the case where $\sum_{n=1}^{\infty} K(n)<1$, the definition of the partition function $Z^{\gb,\go}_{N,h}$ is unchanged if the renewal is replaced by 
 a recurrent one with inter-arrival law given by $K(n)/(\sum_{m=1}^{\infty} K(m))$
 and $h$ is replaced by $h+\log (\sum_{m=1}^{\infty} K(m))$.

\medskip 
 
 \begin{rem}\label{laremark}
 The expression that we gave for the partition function of the disordered system differs substantially from the one usually found in the literature.
However, if we set $\tilde \go^{\gb}_n:= \log (1+\gb \go_{n})$ we can rewrite it in a more usual way
\begin{equation}\label{exponota}
  Z^{\gb,\go}_{N,h} = \bE\left[ \exp\left( \sum_{n=1}^N (\tilde \go^{\gb}_n+h) \delta_n \right) \right].
\end{equation}
Our reason for using a different notation is explained in Section \ref{stabloblo}.
\end{rem}
\subsection{Free energy and comparison with the annealed model.} 
 
  An important quantity that encodes a lot of information about the asymptotic behavior (by this we mean in the limit when $N$ becomes large) of  the renewal $\tau$ under 
  the measure $\bP^{\gb,\go}_{N,h}$ is the free energy per monomer, which is defined as the asymptotic growth rate of the partition function
    \begin{equation}
    \tf(\gb,h) := \lim_{N \to \infty} \frac{1}{N} 
    \log Z^{\gb,\go}_{N,h} \stackrel{\bbbP- a.s.}{=} \lim_{N \to \infty} \tfrac{1}{N} 
    \E\left[ \log Z^{\gb,\go}_{N,h} \right]<\infty.
    \end{equation}  
  The existence and the non--randomness of the limit is a well established fact, we refer to \cite[Theorem 4.1]{cf:GB} for a proof. 
   
   \medskip
   
 The reader can check that 
 $\tf(\gb,h)$ is non-negative, and that $h\mapsto \tf(\gb,h)$ is non-decreasing and convex (as a limit of non decreasing convex functions). 
By exchanging  limit and  derivative, as allowed by convexity, we obtain that the derivative of $\tf$ w.r.t. $h$  corresponds to the asymptotic contact fraction
   \begin{equation}\label{contacts}
    \partial_h \tf(\gb,h):= \lim_{N\to \infty} \frac{1}{N}\bE^{\gb,\go}_{N,h}\left[\sum_{n=1}^N \delta_n\right],
   \end{equation}
   as soon as the derivative exists. Note that by convexity we know that  $\partial_h \tf(\gb,h)$ is defined for all but at most countably many values of $h$, 
   but more advanced results proved in \cite{cf:alea} (for $h\neq h_c(\gb)$) \and \cite{cf:GT05} (for $h=h_c(\gb)$) ascertains that the derivative
   exists everywhere except when $\gb=h=0$ and $\alpha>1$ (see \eqref{puresys}).
   
   \medskip
   
  \noindent If one sets 
  \begin{equation}
   h_c(\gb):= \inf\{ h\in \bbR \ | \  \tf(\gb,h)>0 \},
  \end{equation}
then in view of \eqref{contacts}, $h_c(\gb)$ separates a phase where the contact fraction is vanishing ($h<h_c(\gb)$, the delocalized phase),
from another where it is positive ($h>h_c(\gb)$, the localized phase). Very soft arguments exposed below are sufficient to show that 
this phase transition really occurs, that is that $h_c(\gb) \notin \{-\infty,\infty \}$.

\medskip

\noindent For the homogeneous case, the free energy $\tf(h):= \tf(0,h)$ can be computed explicitly (see \cite{cf:GB}) 
    \begin{equation}
     \tf(h)=\begin{cases}
             0 \quad &\text{ if } h\le 0,\\
             g^{-1}(h) \quad &\text{ if } h>0,
             \end{cases}
           \end{equation}
where $g$ is defined on $\bbR_+$ by
$$g(x):=-\log \left( \sum_{n=1}^{\infty} e^{-nx} K(n) \right).$$
For $\alpha\ne 1$, using some Tauberian theorems, Assumption \eqref{Kt} entails that 
            \begin{equation}\label{puresys}
             \tf(h) \stackrel{h\to 0+}{\sim} C(K)h^{\max(\ga^{-1}, 1)},
            \end{equation} 
    where $C(K)$ is an explicit (see \cite[Theorem 2.1]{cf:GB}) function of the renewal function $K$. 
    When $\alpha=1$ the same result holds with a slowly varying correction in front of the power of $h$.     We refer to the exponent $\max(\ga^{-1}, 1)$ appearing in \eqref{puresys} as the critical exponent associated to the free energy.

    \medskip

 To try to understand the behavior of the disordered pinning model, it is tempting to use comparison with the homogeneous one.
 Making use of Jensen's inequality and \eqref{star} we obtain 
       \begin{equation}\label{raging}
        \E[\log Z^{\gb,\go}_{N,h}] \leq \log \E[Z^{\gb,\go}_{N,h}] = \log(Z_{N,h}), 
       \end{equation} 
         and hence 
          \begin{equation}\label{anne}
           \tf(\gb,h) \leq \tf(h) \quad \text{ and } h_c(\gb)\ge 0.
          \end{equation} 
On the other hand, some other convexity considerations (see \cite[Proposition 5.1]{cf:GB}) yield
\begin{equation}\label{hilde}
     \tf(\gb,h) \geq \tf(h + \log \bbE[1+\gb \go_1] ) \quad \text{ and } h_c(\gb)\le - \bbE[\log(1+\gb \go_1)].
\end{equation}

\subsection{Harris criterion, second moment and stable laws} \label{stabloblo}
 
While \eqref{hilde} is never sharp (see \cite{cf:AS})
the question whether $h_{c}(\gb)$ is equal to zero is a much more subtle one and is very much related  to the question of disorder relevance:
\begin{center}
``Does the introduction of a disorder of small amplitude (small $\gb$)\\
implies a change of the critical behavior of the system ?"
\end{center}
More precisely the question can decomposed in two points:
\begin{itemize}
 \item [(A)] Does the critical point of the disordered system coincide with the one obtained after averaging 
 \eqref{raging}? ( With our conventions: is $h_c(\gb)=0$? )
  \item [(B)] If at the vicinity of the critical point we have 
$$\tf(\gb,h_c(\gb)+u) \approx u^{\nu},$$
does $\nu$ coincide with the exponent of the pure system $\max(\ga^{-1}, 1)$?
\end{itemize}
These questions received a lot of attention from the mathematical community since the publication of heuristic predictions made by Derrida \textit{et al.} \cite{cf:DHV}
based on an interpretation of the Harris criterion \cite{cf:Harris}.

\medskip

In substance, the argument in \cite{cf:DHV} is based on the following observation: if one considers the disordered system at the pure critical point $h=0$,
then the variance of the partition function diverges (exponentially with the size of the system) for any $\gb>0$ if the return exponent is larger or equal to $1/2$ 
and remains bounded for small values of $\gb$ if
$\alpha$ is strictly smaller than $1/2$. From these observations and some heuristic computations, they conclude that 
\begin{itemize}
 \item  [(1)] When $\gb$ is small and $\alpha<1/2$ the critical point $h_c(\gb)$ is equal to $0$ which is that of the pure system,
 and furthermore, the critical exponent associated to the free energy is equal  to $\alpha^{-1}$.
 \item [(2)] When $\alpha\ge 1/2$, there is a shift of the critical point ($h_c(\gb)>0$) for all values of $\gb$, 
 which is of order $\gb^{\frac{2\alpha}{2\alpha-1}}$ when $\gb$ is small.
 \end{itemize}
These predictions were confirmed when $\alpha<1/2$ \cite{cf:A06, cf:T08, cf:Lmart}, $\alpha>1/2$  \cite{cf:DGLT07,cf:AZ08} and $\alpha=1/2$
\cite{cf:GLT08, cf:BL}, in the case where $\go$ has a finite second moment.

\medskip

From the proof heuristics, one is led to believe that the assumption $\bbE[\go_n^2] < \infty$ is not only a technical detail, and that 
considering disorders with heavier tail might considerably change the picture of disorder relevance.

\medskip

In this paper we decide to consider the case where the $\go_n$'s are heavy tailed, and more precisely are in the domain of attraction of 
a $\gamma$-stable law with $\gamma\in(1,2)$.
To keep things simples we assume that there exists a constant $C_{\bbP}$ such that
\begin{equation}\label{defgamma}
    \bbP[\go_n\ge x]\stackrel{x\to \infty}{\sim} C_{\bbP}x^{-\gamma}, \quad \gamma\in (1,2).
\end{equation}
This justify our unorthodox choice for the writing of the partition function (underlined in Remark \ref{laremark}):
writing things in the usual way, we would end-up with an exponent $\gamma$ which depends on $\gb$ which would be unpractical.

\begin{rem}\label{slova}
All the results presented in this paper, would also extend to the case where one allows the presence of a slowly varying function instead of the constants $C_K$ and $C_{\bbP}$ in the tail distribution of the renewal process \eqref{Kt}
or of the environment \eqref{defgamma}.
We made the choice of a more restrictive assumption to simplify the notation. 
\end{rem}

      \subsection{Results}

Our main achievement is to show that when environment with heavier tail is considered, Harris criterion is not valid anymore:
 the value of $\alpha$ that separates disorder relevance and irrelevance is no longer equal to $1/2$ but to $1-\gamma^{-1}$.
 In all the results stated below we assume that \eqref{Kt} and \eqref{defgamma} holds.

 \medskip
 
 Our first result states that when $\alpha< [1-\gamma^{-1}]$, the critical point and critical exponent are the same as the one of the pure system. 

 \begin{theorem}\label{diso1}
 If $\alpha< [1-\gamma^{-1}]$, 
 there exists $\gb_0 > 0$ (depending on the renewal function $K$ and of the distribution of $\go$) 
 such that for all $\gb \in (0,\gb_0)$,  we have 
 $h_c(\gb)=0$ and 
\begin{equation}\label{freene}
  \lim_{h\to 0+} \frac{\log \tf(\gb,h)}{\log h}=\frac{1}{\alpha}.
\end{equation}
 \end{theorem}

Our second result deals with the relevant disorder case: we show that  when  $\alpha>[1-\gamma^{-1}]$ 
the critical point is shifted for every value of $\gb$.

 \begin{theorem}\label{diso}
 If $\alpha>[1-\gamma^{-1}]$, then for all $\gb < \gb_0$ we have $h_c(\gb)>0$.
 Furthermore we have the following estimate on the critical point shift 
 \begin{equation}\label{shiftee}
  \lim_{\gb\to 0+} \frac{\log h_c(\gb)}{(\log \gb)}=\begin{cases} \frac{\alpha\gamma}{(\alpha-1)\gamma+1}, \quad & \text{if } \alpha\in (1-\gamma^{-1},1) \\
  \gamma  \quad &\text{if } \alpha\ge 1.                                                   \end{cases}
\end{equation}
     \end{theorem}
Note that in the case where $\go$ has finite second moment both \eqref{freene} and \eqref{shiftee} hold (with proofs given respectively in \cite{cf:A06, cf:T08, cf:Lmart} and 
 \cite{cf:DGLT07, cf:AZ08}) with $\gamma$ replaced by $2$.
 
 \medskip

 \begin{rem}
 We prove in fact more quantitative upper and lower bounds for \eqref{freene} and \eqref{shiftee}, but our upper and lower bounds do not match.
 We believe that for $\alpha< [1-\gamma^{-1}]$,  $\gb \in (0,\gb_0)$,  the annealed bound \eqref{anne} is asymptotically sharp, that is
 $$\tf(\gb,h)\stackrel{h\to 0+}{\sim} \tf(0,h).$$
 When $\alpha>[1-\gamma^{-1}]$ we believe, that, similarly to what occurs in the $L_2$ case \cite{cf:CTT} we have 
 $$h_c(\gb)\stackrel{\gb \to 0_+}\sim c \gb^{\frac{\alpha\gamma}{(\alpha-1)\gamma+1}}.$$
 \end{rem}

 \subsection{Open questions and conjectures}
  
  \subsubsection{The marginal case}
  
  An important observation about our results is that they do not  solve the marginal case $\alpha=[1-\gamma^{-1}]$.
  In a forthcoming companion paper \cite{cf:proceeding}, we show that when $\alpha=[1-\gamma^{-1}]$ when \eqref{Kt} is satisfied disorder is also relevant,
  a result that bears some similarity with that proved in \cite{cf:GLT08}. 
  
  \medskip

  However this does not completely solves the problem of disorder relevance:
  for instance when \eqref{defgamma} holds, we would like to find a necessary and sufficient condition on $K$ (assuming only regular variations) 
  for the occurrence of 
  a critical point shift at every temperature similar to the one proved in \cite{cf:BL} under the assumption of second moment.
  
  \medskip
  
  Heuristic computations suggest the following picture:
  
  \begin{conjecture}\label{c:conjectos}
  Assuming that Assumption \eqref{defgamma} holds and that the inter-arrival law $K(\cdot)$ is regularly varying, the following equivalence holds
    \begin{equation}\label{conjectos}
   \left\{ \forall \gb>0, \ h_c(\gb)>0 \right\} \ \Leftrightarrow  \ \left\{ \sum_{n\ge 1} \bP[n\in \tau]^{\gamma} \right\}. 
 \end{equation}
  \end{conjecture}

While it seems plausible that, with a consequent amount of work, the techniques developed in \cite{cf:BL} could be adapted to prove one 
side of the implication (that is, that the proposition on the r.h.s.\ implies that on the l.h.s.\ ), 
the other direction seems to be much more challenging with the techniques we have at hand.

\subsection{Smoothing of the phase transition}

It is a general paradigm that the presence of disorder tends to make the free energy curve smoother at the vicinity of the critical point (see e.g. \cite{cf:AW} for a 
celebrated result of this kind for the random field Ising model).
For disordered pinning models in particular the first result of this type was proved in \cite{cf:GT05} and generalized in \cite{cf:CdH}.
This last generalization applies without restriction to our setup and we have \cite[Theorem 1.9 ]{cf:CdH}, for every $\gb>0$ (and every $\gamma$ and $\alpha$),
there exists a constant $C_{\gb}>0$ such that for all $u\in[0,1]$, 
\begin{equation}
 F(\gb,h_c(\gb)+u)\le C_{\gb} u^2.
\end{equation}
There are various reasons to believe that an heavy-tailed environment should make the free energy curve even smoother than quadratic at criticality.
More precisely
  \begin{conjecture}\label{supersmooth}
Assuming that \eqref{Kt} and \eqref{defgamma} holds, for every $\gb>0$ 
there exists $C_{\gb}>0$ such that for all
$u\in[0,1]$,
    \begin{equation}\label{conjectos2}
  F(\gb,h_c(\gb)+u)\le C_{\gb} u^{\frac{\gamma}{\gamma-1}}.
 \end{equation}
  \end{conjecture}
   
A first justification for this conjecture is that for relevant disorder, the free energy of the disordered system should be smoother than that of the pure system.
Thus the critical exponent for the disordered system should be larger than $\alpha^{-1}$ for every $\alpha>1-\gamma^{-1}$.
Perhaps a more convincing one is that the proofs in \cite{cf:GT05, cf:CdH} are based on localization strategies which 
take advantage of rare fluctuations of the environment.
With heavy tailed $\go$'s, this strategy could in principle be improved using the presence of larger fluctuations.
However, there are serious technical obstacle to transform these heuristics into a proof.

\subsubsection{Scaling limits}

Important efforts have been recently performed in the community to understand in which way systems with relevant disorder scale to 
continuous limits by tuning the intensity
of the disorder to zero while the system grows \cite{cf:CSZ2, cf:AKQ}. So far, to our knowledge only the case of 
disorder with finite second moment has been considered, and 
limits have been found to be related to Gaussian multiplicative chaos.

\medskip

Here, due to the different nature of the noise, the scaling limit should no longer be Gaussian but should involve some Levy noise, and the results we obtain also suggest that  
the appropriate scaling should be different. 

\begin{conjecture}\label{levyscaling}
 When $\alpha\in (1-\gamma^{-1},1)$, given $\hat \gb>0$ and $\hat h\in \bbR$, one sets 
 \begin{equation}\begin{cases}
   \gb_N&:= \hat \gb N^{1-\alpha-\gamma^{-1}},\\
    h_N&:= \hat h N^{-\alpha},
 \end{cases}\end{equation}
then the sequence $(N^{1-\alpha} Z^{\gb_N,\go}_{N,h_N})_{N}$ converges in law to a non-degenerate random variable.
\end{conjecture}

The factor $N^{1-\alpha}$ is present to compensate for the cost of the conditioning $N\in \tau$ and would not be present if one considered another type of boundary condition.
We believe that the limit could be expressed as a multiplicative chaos over a Levy Noise similar to the Levy multiplicative chaos defined in \cite{cf:RSV}
but defined on a state of paths (see \cite{cf:Shamov} for a general definition of chaos in the Gaussian setup).

When $\alpha>1$, $h_N=h N^{-1}$ and $\gb_N:=\hat \gb N^{-\gamma^{-1}}$, the sequence  $Z^{\gb_N,\go_N}_{N,h_N}$ should also converge to some non-degenerate distribution. 
While the rigorous  proof in the case $\alpha>1$ does not seem to present a big challenge, 
Conjecture \ref{levyscaling} on the other hand could prove to be quite daring as most of the tools used in \cite{cf:CSZ2} 
do not seem to adapt to the Levy case, and some new  ideas should be developed.  
  
 \subsection{About the proofs}
 
 \subsubsection{Disorder irrelevance and upper bound on the critical point shift}
 
 The proof of Theorem \ref{diso1} requires a new method as all the proof of results of this type in the literature rely on controlling the second moment 
 which is infinite in our case. What we do instead is to try to control the moment of order $p$ for some $p$ in the interval $(1,\gamma)$. 
 
 \medskip
 
 The problem that arises then is that while integer moments of partition functions have generally a nice expression involving several 
 replicas of the system, this is not the case for fractional moments. To obtain suitable upper-bounds on the 
 fractional moments, we first rewrite the problem as the estimation of the $(p-1)$-th moment of a partition of the system where
 the environment has been tilted along a quenched renewal trajectory.
 
 \medskip
 
To obtain an upper bound on this modified partition function, we then perform an adequate partial
annealing and a decomposition of the partition function which takes into account the high-cost
of making long jumps. These computations require a fine intuition of the mechanism that yields self-averaging in the partition function.

\medskip

To control the value of $h_c(\gb)$, like in \cite{cf:AZ08,cf:T08}, we try to control the rate of explosion of the partition function for small values of $\gb$.
Depending on the value of $\alpha$, we perform this by controlling either the second moment 
($\alpha\ge 1/2$) or a fractional moment $\alpha< 1/2$ of a partition function with truncated 
environment.

\subsubsection{Disorder relevance}

In the case of relevant disorder, we estimate the $q$-fractional moments of the partition function for some $q\in (0,1)$ with the help of a coarse graining procedure 
combined with a change of measure argument. This method has been introduced  \cite{cf:GLT07} and has been improved several times \cite{cf:GLT08, cf:BL}. 
The underlying idea is to introduce a penalization for atypical environments, that is environments which have small probability but give an important contribution to 
the partition function $\bbE[Z^{\gb,\go}_{N,h}]$.

\medskip

The important novelty added in this paper is that we do not penalize the environment for which the empirical mean of the $\go$ is too large like in the $L^2$ case,
but we choose to penalize environments for which the extremal values of $\go$ are large, as heavier-tailed distributions tend to make these values meaningful.

\subsection{Organization of the paper}  
  In Section \ref{technicos} we introduce a couple of technical results that which are required for the proofs.
  Sections \ref{proofirel} and \ref{fracof} are dedicated to the proof of Theorem \ref{diso1}, this is the most novel part of the paper where an original method is introduced
  to treat disorder irrelevance in the absence of second moment. In Section \ref{upperboundcps} the upper bound on $h_c(\gb)$ present in Equation \eqref{shiftee} are proved, 
  partly using the method used to prove disorder irrelevance.
  Finally in Section \ref{rerel} we prove the lower bounds from  Equation \eqref{shiftee} which completes the proof of Theorem \ref{diso}.

\section{A few technical tools}\label{technicos}

\subsection{Estimate on probability of visiting a given point}

Set $u(n)=\bP[n \in \tau]$. 
Under our power-law tail assumption the asymptotic behavior of $u(n)$ is well identified.

\begin{lemma}\label{lem:don}
If the assumption \eqref{Kt} holds, then there exists $C'_K$ (depending on the renewal function $K$) such that 
\begin{equation}\label{doney}
 u(n)\sim \begin{cases} C'_K n^{\alpha-1}, \quad & \text{ if } \alpha\in(0,1),\\
 C'_K (\log n)^{-1}, \quad & \text{ if } \alpha = 1,\\
                        C'_K \quad & \text{ if } \alpha>1.
          \end{cases}
\end{equation}
\end{lemma}
The case $\alpha>1$ comes from \cite{cf:doney}, and the case $\alpha>1$ is a consequence of the renewal theorem 
(we refer to \cite[Theorem 2.2(3)]{cf:GB} for full references, including the case $\alpha=1$).

\subsection{Finite volume criteria}

To obtain a lower on the free energy, it is sufficient to obtain a bound on the partition function of finite size.
Indeed we can observe that 
 $\bbE[\log Z^{\gb,\go}_{N,h}]$ is a super-additive sequence \cite[Proposition 4.2]{cf:GB} and thus, for every $N$:
 \begin{equation}
   \tf(\gb,h)\ge \frac{1}{N}\bbE[\log Z^{\gb,\go}_{N,h}].
 \end{equation}
  
 However, in our lower bound computations, it will be much more convenient for us to work with the free-boundary partition function where the constraint 
  $\{N\in \tau\}$ is dropped, that is 
 \begin{equation*}
    Z^{\gb,\go,\cf}_{N,h} := \bE\left[\prod_{n\in \tau \cap [1,N]} e^h(\gb\go_n+1)\right].
 \end{equation*}
  
 Note that  $Z^{\gb,\go,\cf}_{N,h}$ compares well with $Z^{\gb,\go}_{N,h}$ (cf. \cite[Equation (4.25)]{cf:GB}); indeed, this quantity 
 also provides a lower bound on the free energy with the loss of a $\log$ factor (for a proof see \cite[Proposition 2.6]{cf:albezhou}).
 
 \begin{lemma}
 There exists a constant $C(\gb)$ such that for any value of $N$ and $h$,
    \begin{equation}\label{finitfree}
   \tf(\gb,h)\ge \frac{1}{N}\bbE[\log Z^{\gb,\go,\cf}_{N,h}]-C(\gb) \frac{\log N}{N}. 
    \end{equation}
\end{lemma}

     \section{Disorder irrelevance: the proof of Theorem \ref{diso1}}\label{proofirel}
     
The idea used in the proof is similar to the one introduced in \cite{cf:Lmart}:
if at the pure critical point the partition function behaves like its average, 
it implies that the measure $\bP^{\gb,\go}_{N,0}$ is in a sense close to the original one $\bP$.
We want to use this information to prove that the expected number of contacts at criticality is large, from which we get a 
lower bound on the partition function $Z^{\gb,\go,\cf}_{N,h}$ at positive times, and finally conclude using \eqref{finitfree}.
Using this procedure, we can prove the following result:

\begin{proposition}\label{irrelevex}
 Assume that \eqref{defgamma} holds and that $\alpha<[1-\gamma^{-1}]$.
 There exists $\gb_0$ such that for all $\gb\in[0,\gb_0]$, there exists $C_\gb$ (that may depend also on the inter-arrival distribution and the distribution of $\go$) such that
 \begin{equation}\label{eq:dalowbound}
\forall h\in[0,1], \quad \tf(\gb,h)\ge C_{\gb}  \frac{h^{\alpha^{-1}}}{|\log h|^{(\alpha-1)/\alpha}}.
 \end{equation}
\end{proposition}

Contrarily to what we do in  \cite{cf:Lmart}, we do not use the convergence of the partition function (as a martingale), but find a more efficient way to
use uniform integrability in order to extract quantitative statements. Then we use the same technique to prove upper-bounds on $h_c(\gb)$ in the 
disorder relevant case.

  \subsection{Decomposition of the proof} \label{decompex}

In this short section, we show how Proposition \ref{irrelevex} follows from two key statements.
The most important one, whose proof is detailed in Section \ref{fracof}, is that some non-integer moments of order $p>1$
of the partition function are uniformly bounded in the size of the system.

  \begin{proposition}\label{fractionalmoment}
   For any $p\in(1, \gamma)$ and for $\gb\le  \gb_0(p)$, we have 
   \begin{equation}
    \sup_{N\ge 0} \bbE\left[ (Z^{\gb,\go,  \cf}_{N,0})^p \right]<\infty.
    \end{equation}
    \end{proposition}    
    
The second result, whose short proof is detailed in Section \ref{simplex} uses this bound to show that typical events for 
$\bP$ cannot be atypical for $\bP^{\gb,\go,\cf}_{N,0}$. 
 \begin{lemma}\label{simplifiex}
Given  $p>1$, if $\bbE\left[ (Z^{\gb,\go,\cf}_{N,0})^p \right]=: M < \infty$, there exists $\delta = \delta(M,p)>0$ such that 
for any event $A$
\begin{equation}
 \left\{  \ \bP[A]\ge 1-\delta  \ \right\}\quad \Rightarrow  \quad  \left\{ \ \bbE\left[ \bP^{\gb,\go,\cf}_{N,0}[A] \right]\ge \delta \right\}.
\end{equation}
\end{lemma}

 Set $\cN_N(\tau):= \#([1,N]\cap \tau)$ be the number of renewal points in the interval $[1,N]$.
Given $\gep>0$, we apply Lemma \ref{simplifiex} to the event
 \begin{equation*}
  A_{N,\gep}:=\left\{ \tau \, : \, \cN_N(\tau)  \ge \gep N^{\alpha} \right\}.
 \end{equation*}  
As a consequence of the convergence of $(n^{-1/\alpha}\tau_{\lceil nt \rceil})_{t\ge 0}$ to an $\alpha$-stable subordinator (see  \cite[Chapter XVII]{cf:Feller}),
$N^{-\alpha}\cN_N(\tau)$ converges in law to a random variable (the first hitting time of $[1,\infty)$ for this subordinator) 
whose distribution has no atom at zero. In particular 

 \begin{equation}\label{crocodiles}
 \lim_{\gep\to 0}\limsup_{N\to \infty} \bP(A_{N,\gep})=0.
 \end{equation}

 \begin{proof}[Proof of Proposition \ref{irrelevex}]
 We have now all the ingredients to prove \eqref{eq:dalowbound}.
 We fix $p\in(1,\gamma)$ arbitrarily, consider $\gb\le  \gb_0(p)$, set $M=M(\gb):= \sup_{N\ge 0} \bbE\left[ (Z^{\gb,\go,\cf}_{N,0})^p \right]$
 and choose $\gep_0$ and $N_0$ such that  
 $\bP(A_{N,\gep_0})\ge 1-\delta(M,p)$, for all $N\ge N_0$.
 From Lemma \ref{simplifiex}, we have 
 \begin{equation}\label{deltouze}
  \bbE\left[ \bP^{\gb,\go,\cf}_{N,0}[A_{N,\gep_{0}}] \right]>\delta.
 \end{equation}
 Using the convexity of the function $h \mapsto \log Z^{\gb,\go,  \cf}_{N,h}$, we observe that for some constant $C$
 \begin{multline}
 \log Z^{\gb,\go,\cf}_{N,h}\ge  \log Z^{\gb,\go,\cf}_{N,0} + h\partial_u \log Z^{\gb,\go,\cf}_{N,u}|_{u=0}\\
 \ge  \log  P[\tau_1> N]+ h \E^{\gb,\go,\cf}_{N,0}[ \cN_N(\tau)] \ge  h \gep_0 N^{\alpha}  \bbP^{\gb,\go,\cf}_{N,0}[A_{N,\gep_{0}}]- \alpha[\log N] -C.
\end{multline}
Taking  the expectation and using \eqref{deltouze} and \eqref{finitfree} we obtain that for some $N_0$ sufficiently large, for any $h>0$ and $N\ge N_0$ we have 
 
 \begin{equation}
  \tf(\gb,h) \ge  h  N^{\alpha-1}\gep_0 \delta -C' \frac{\log N}{N},
 \end{equation}
and we conclude by taking $N=C'' \left(h^{-1}|\log h|\right)^{\alpha^{-1}}$, for $C''$ sufficiently large.
\end{proof}

\subsection{Proof of Lemma \ref{simplifiex}}      \label{simplex}

 We have 
 \begin{multline}
 \bbE\bP^{\gb,\go,\cf}_{N,0}\left[A^{\cc}\right]= \bbE\left[ \frac 1 {Z^{\gb,\go,\cf}_{N,0}} \bE\left[A^{\cc}\prod_{i=1}^N(1+\gb \go_n\delta_n)\right] \right]\\
 \le 2 \bbE \bE\left[A^{\cc}\prod_{i=1}^N(1+\gb \go_n\delta_n)\right]+ \bbP\left[Z^{\gb,\go,\cf}_{N,0}< 1/2\right]
 \\=2\bP\left[A^{\cc}\right]+\left(1- \bbP\left[Z^{\gb,\go,\cf}_{N,0}\ge  1/2\right]\right).
  \end{multline}
  To conclude we use the following estimate for $\theta=1/2$.
  
     \begin{lemma}\label{superzygmund}
      If $X$ is a positive random variable and $p>1$, we have 
      \begin{equation}
       \bbP\left[X\ge \theta \bbE[X]\right]\ge (1-\theta)^{\frac{p}{p-1}}\frac{\bbE[X]^{\frac{p}{p-1}}}{\bbE[X^p]^{\frac{1}{p-1}}}.
      \end{equation}
\end{lemma}
 Using the above result for $X= Z^{\gb,\go,\cf}_{N,0}$ and $\theta=1/2$ together with the assumption on $A$ this yields 
  \begin{equation}
    \bbE\bP^{\gb,\go,\cf}_{N,0}\left[A^{\cc}\right]\le 2\delta + 1 - 2^{-\frac{p}{p-1}}M^{-\frac{1}{p-1}} \le 1-\delta, 
   \end{equation}
provided $\delta$ is chosen sufficiently small.

\begin{proof}[Proof of Lemma \ref{superzygmund}]
   Using H\"older's inequality (with $p'=\frac{p}{p-1}$), we get 
      \begin{equation}
       \bbE[X]\le \theta\bbE[X] + \bbE\left[X\ind_{\{X\ge \theta\bbE[X]\}}\right]\le \theta\bbE[X] + (\bbE[X^p])^{1/p}(\bbP\left[X\ge \theta\bbE[X] \right])^{1/p'}.
      \end{equation}
    \end{proof}

\section{Bounding the fractional moments: Proof of Proposition \ref{fractionalmoment}}\label{fracof}
  
  \subsection{Decomposing the proof}
  
  By monotonicity of $p\mapsto \bbE[ (Z_N)^p]^{1/p}$,  it is sufficient to treat the case $p\in \left(\frac{1}{1-\alpha}, \gamma\right)$.
     We set $q=p-1$, and we assume during the whole proof that
    \begin{equation}\label{asumpq}
   \frac{\alpha}{1-\alpha}<q<\gamma-1.
    \end{equation}
While it is clear from the assumption \eqref{defgamma} that for $p\ge \gamma$ the moment of order $p$ is equal to infinity, 
it is not obvious at this stage why we also need a lower bound on $q$.
This will appear in the course of the proof when using coarse graining arguments.

\medskip

Our proof goes as follow, first (Section \ref{sizebias}) we rewrite the $p$ moment of the partition function as the $q$ 
moment of a different partition function involving an extra quenched copy of the renewal and a tilted environment. 
We also perform some partial annealing to simplify the expression which is obtained.

\medskip

In a second step (Section \ref{corsgrain}) we perform a decomposition of this new partition function, based on the classic inequality $(\sum a_i)^q\le \sum a^q_i$.
This helps us to reduce our proof to the estimate of (the $q$ moment of) the partition function of a system of finite size.

\medskip

Finally we use a change of measure technique (Section \ref{changameas}) to show that this last partition function is small. 
      
      \subsection{Rewriting the partition function using size-biasing} \label{sizebias}
      
    To simplify the quantity we have to bound, we decide to rewrite is as $\tilde \bbE_N \left[\left( Z_{N,0}^{\gb,\go,\cf}\right)^{q}\right]$,
    where $\tilde \bbP_N$ is the probability defined by
    \begin{equation}\label{deftildepn}
    \frac{\dd \tilde \bbP_N}{ \dd \bbP}(\go)= Z_{N,0}^{\gb,\go,\cf}.
    \end{equation}
    The partition function having expectation one under $\bbP$ it defines indeed a probability density.
    The reason for which this consideration might be useful is that more techniques are available to control $p$ moments for $p$ in the interval $(0,1)$ than 
    for $p$ in $(1,2)$ (we refer for example to the key role of the inequality \eqref{eq:subadt} in our proof). 
    A first important thing to do however, is to re-express $\tilde \bbP_N$ in a form which will be more adequate to perform computations.
    This representation of the sized biased measure which we present below, sometimes called "spine representation" in the case of branching structures (see 
    the recent monograph
    \cite{cf:Shi} and references therein) 
    is classical, and has been used several times in the framework of polymer measures (see e.g. \cite{cf:Bir}).
    
    \medskip

    Let $\tau'$ be an independent copy of the renewal $\tau$ (we  denote its law by $\bP'$). 
    We notice that, for a fixed realization of $\tau'$, the quantity $\prod_{n=1}^N (1 + \gb \go_{n} \delta_n)$ averages to one under 
       $\bbbP$, and thus can be considered as a probability density. Given a realization of $\tau'$, we introduce the probability measure $\bbP_{\tau',N}$ whose density with respect to 
        $\bbbP$ is given by
        \begin{equation}\label{souz}
         \bbP_{\tau',N}(d \go) = \prod_{n=1}^N (1 + \gb \go_{n} \ind_{\{n \in \tau'\}}) \bbbP(d \go). 
        \end{equation} 
     With these notations, we can write  
      \begin{equation}\label{siz}
       \bbE\left[\left(Z_{N,0}^{\gb,\go,\cf}\right)^{p}\right] = \bbE \bE'\left[\prod_{n=1}^N (1 + \gb \go_{n} \ind_{\{n \in \tau'\}}) 
       \left(Z_{N,0}^{\gb,\go,\cf}\right)^{q}\right]=\bE' \left[ \bbE_{\tau',N} \left[  \left(Z_{N,0}^{\gb,\go,\cf}\right)^{q}\right] \right].
      \end{equation} 
        
        Note that under $\bbP_{\tau',N}$, the random variables $\go_{n}$ are still independent but
        they are no longer identically distributed, since the law of $(\go_{n})_{n \in [1,N] \cap \tau'}$ has been tilted by the quantity $(1 + \gb \go_n)$.
        However we can construct an environment $\hat \go$ of law $\bbP_{\tau',N}$ using two IID environment as follows:
      \begin{itemize}
       \item   [(1)]    First we set $(\tilde \go_{n})_{n \geq 1}$ to be an IID tilted environment; namely, 
        all $\tilde \go_n$'s are IID, and 
        \begin{equation}\label{defitilt}
        \tilde \bbP[\tilde \go_1 \in \dd x]:=(1+\gb x) \bbP[\go_1 \in \dd x].
        \end{equation}

       \item [(2)]
        Given a realization of $\tau', \go, \tilde \go$, we define the sequence $\hat{\go}$  in 
          the following way
          \begin{equation}\label{defahah}
           \hat{\go}_{n} := \hat{\go}_{n}(\go,\tilde \go,\tau',N) = \go_{n} \ind_{\{n \notin \tau'\cap [1,N] \}} + \tilde \go_{n} \ind_{\{n \in \tau'\cap [1,N]\}}.
          \end{equation}
      \end{itemize}
  With this notation we have for every realization of $\tau'$ 
          $$\bbP_{\tau',N}\left[ \go \in \cdot \right]= \tilde \bbP \otimes \bbP\left[ \hat \go \in \cdot \right],$$
 and Fubini's identity yields
              \begin{equation*}
         \bE\left[\bbE_{\tau',N}\left[\left(Z_{N,0}^{\gb,\go,\cf}\right)^{q}\right]\right]  = 
                          \tilde\bbE \otimes  \bbE \otimes \bE'\left[\left(Z_{N,0}^{\gb,\go,\cf}\right)^{q}\right].           
     \end{equation*}
                   Using Jensen's inequality with respect to the measure $\bbE$ (recall that we have chosen $q\in(0,1)$): 
                   \begin{equation}
                      \tilde\bbE \otimes  \bbE \otimes \bE'\left[\left(Z_{N,0}^{\gb,\go,\cf}\right)^{q}\right]\le 
                       \tilde \bbE \bE' \left[ \left(  \bbE \left[Z_{N,0}^{\gb,\go,\cf}\right]\right)^{q} \right]
                      =  \tilde \bbE   \bE' \left[ (Z^{\gb,\tilde\go}_{N}[\tau'])^q \right]
                   \end{equation}
where, for a given realization of $\tau'$, we defined 
    $$Z^{\gb,\tilde\go}_{N}[\tau']:= \bE\left[ \prod_{n=1}^N(1+\gb \tilde \go \ind_{\{n\in\tau\cap \tau'\}}) \right].$$
   To conclude this section, let us thus record that we reduced the proof of Proposition \ref{fractionalmoment} to the proof of the  following statement.
    \begin{lemma}\label{conclouse}
    For any $q\in\left( \frac{\alpha}{1-\alpha}, \gamma-1 \right)$, and $\gb\le \gb_0(q)$ we have 
    \begin{equation}
     \sup_{N\ge 0} \tilde \bbE   \bE' \left[ (Z^{\gb,\tilde\go}_{N}[\tau'])^q \right]<\infty.
    \end{equation}
    \end{lemma}

   \begin{rem} 
  We used Jensen inequality here to obtain a more tractable expression to estimate (in particular we are back with only one IID environment).
  A way to justify that this step does not make us lose much is that when disorder is irrelevant (which is what we aim to prove)
  the disorder is self-averaging.
  On the other hand,  we cannot simply use Jensen's inequality for $\tilde \bbE$ as $\tilde \go$ has infinite mean.
       Applying Jensen to $\bP'$ is also not optimal as in some cases, one could prove that 
         under our assumptions,  $\tilde \bbE\left[(\bE'[Z^{\gb,\tilde\go}_{N}[\tau']])^q \right]$ diverges.
            
           \end{rem}
           
           \subsection{Coarse graining} \label{corsgrain}

          Our idea  to estimate the $q$-th moment of the partition function $Z^{\gb,\tilde\go}_{N}[\tau']$ it to use a change of measure argument; namely, 
              given a positive function $G$ of $\go$, by H\"older's inqeuality we have
          \begin{equation}\label{davycrockett}
          \tilde \bbE \left[ (Z^{\gb,\tilde\go}_{N}[\tau'])^q \right]\le         \tilde \bbE \left[ (G(\go) Z^{\gb,\tilde\go}_{N}[\tau'])\right]^{q}  
          \tilde \bbE \left[ (G(\go))^{-\frac{q}{1-q}}\right]^{1-q}.  
          \end{equation}
          If we apply this inequality for  $G(\go)=\prod_{i=1}^N g(\go_i)$ with $\tilde \bbE[ g(\go_i) ]=1$, the first term 
          corresponds to the expectation of $\tilde Z$ 
          under a new measure for which the law of $\tilde \go$ has been changed and the second one can be interpreted as a cost for this change of measure.
          
          \medskip
          
          The problem in this approach is that the cost grows exponentially in the size of the system $N$ and the first term, which is always larger than 
          $\bP[\tau_1>N]\approx N^{-\alpha}$ cannot compensate for it. We want to apply this idea in a more subtle way by coupling it with a coarse graining argument. 
          To perform this we decompose the partition function in order to make sure that we apply our change of measure in regions that have a positive density of contacts
          involving the environment $\tilde \go$ (i.e contacts for $\tau\cap \tau'$).
          Hence we decompose $\tilde Z$ according to the location of the large jumps.
          Let us introduce a few notations in order to apply this decomposition.
          
          \medskip

            \noindent    Given  a realization of $\tau'$, for $b>a$ we set
                \begin{equation}\label{defuz}
                u(a,b,\tau')= \ind_{\{a,b \in \tau' \}} \bP[ b\in \tau \text{ and } \tau\cap \tau' \cap (a,b)=\emptyset \ | \ a \in \tau ],
                \end{equation}
                which is the probability that $b$ is the next point in $\tau \cap \tau'$ after $a$. Setting by convention $t_0:=0$,
                   we have 
      
      \begin{multline}
       Z^{\gb,\tilde\go}_{N}[\tau']=\sum_{k\ge 1} \sum_{0<t_1<\dots<t_k\le N} \prod_{i=1}^k u(t_{i-1},t_i,\tau')(1+\gb \tilde \go_{t_i})
       \bP[ (t_i,N]\cap \tau \cap \tau'=\emptyset]\\
       \le  \sum_{k\ge 1} \sum_{0<t_1<\dots<t_k} \prod_{i=1}^k u(t_{i-1},t_i,\tau')(1+\gb \tilde \go_{t_i})=:  
      Z^{\gb,\tilde\go}[\tau'].
       \end{multline}
       In the remaining part of this proof, we will be interested only in bounding the $q$-th moment of  $Z^{\gb,\tilde\go}[\tau']$.
             
         \medskip
         
      We want to express    $Z^{\gb,\tilde\go}[\tau']$ in terms of partition functions where all the gaps in $\tau \cap \tau'$ are smaller than $L$.
          For $L>0$ (meant to be large) and $b>a> 0$ all integers, we define  $Z_{[a,b]}^{L}[\tau']$ as
      the partition function on the interval  $[a,b]$ restricted to trajectories for which $\tau \cap \tau'$ has no gap larger than $L$, that is
\begin{equation}\label{defZll}
Z_{[a,b]}^{L}[\tau'] :=  (1+\gb \tilde \go_{a})\sum_{k\ge 1} \sumtwo{a=t_0<t_1<\dots<t_k=b}{t_{i}-t_{i-1}\le L} 
\prod_{i=1}^{k} u(t_{i-1},t_i,\tau')(1+\gb \tilde \go_{t_i}).
\end{equation}
By convention we set            
       \begin{equation}
 Z_{[a,a]}^{L}[\tau']  = (1 + \gb \tilde\go_{a})\ind_{\{a\in \tau'\}}.
        \end{equation} 
Note that $Z_{[a,b]}^{L}[\tau']=0$ if either $a\notin \tau'$, $b \notin \tau'$ or $\tau'_{i+1}-\tau'_i\ge L$ for some $i$ with $\tau_i\in[a,b)$
and that $Z_{[a,b]}^{L}[\tau'] > 0$ in all other cases.

        Decomposing according to the cardinality, the locations and the lengths of the excursions which are longer than $L$, we get the expression
    \begin{equation}
    (1 + \gb \tilde\go_{0}) Z^{\gb,\tilde \go}[\tau'] = \sum_{k\ge 1}\sum_{(\bt,\bt')\in \cT^{k}_L}
  \left(\prod_{i=1}^{k-1}  Z_{[t_{i},t'_i]}^{L}[\tau']u(t'_{i},t_{i+1},\tau')  \right) Z_{[t_{k},t'_k]}^{L}[\tau']. 
    \end{equation}  
         where
         \begin{equation}
          \cT^{k}_L:= \{ (\bt,\bt')\in \bbN^{2k} \ : \ t_0=0, \ \forall i\in \lint 1,  k \rint,  t'_i\ge t_i \text{ and }  t_{i+1}-t'_i> L  \}.
        \end{equation}
Hence using the inequality 
\begin{equation}\label{eq:subadt}
       \left( \sum_{i\in I} a_{i}\right)^{q} \leq \sum_{i\in I} a_{i}^{q}
       \end{equation}
       valid for any collection of positive numbers, and combining it 
       with the IID nature of the environment, we obtain
       \begin{multline}
      \tilde \bbE \left[ (1 + \gb \tilde\go_{0})^q (Z^{\gb,\tilde \go}[\tau'])^q\right]\\
       \le  \sum_{k\ge 1}\sum_{(\bt,\bt')\in \cT^{k}_L}
  \left(\prod_{i=1}^{k-1}  \tilde \bbE\left[\left(Z_{[t_{i},t'_i]}^{L}[\tau']\right)^q \right] u(t'_{i},t_{i+1},\tau')^q  \right) 
  \tilde\bbE\left[\left(Z_{[t_{k},t'_k]}^{L}[\tau']\right)^{q}\right].
       \end{multline}
       The terms in the sum are zero unless $t_i\in \tau'$ and $t'_i\in \tau'$ for all $i$. We use the spatial Markov property for $\tau'$ 
       and obtain that 
       \begin{multline}\label{lasomcomplexe}
             \bE'\tilde \bbE \left[ (1 + \gb \tilde\go_{0})^q (Z_{N}^{\gb,\tilde \go})^q[\tau']\right]
                     \le  \sum_{k\ge 1}\sum_{(\bt,\bt')\in \cT^{k}_L} \bP'\left[\forall i \in \lint 1,k\rint,\ t_i,t'_i\in \tau' \right] \\
 \times \left(\prod_{i=1}^{k-1} \bE' \tilde \bbE\left[\left(Z_{[t_{i},t'_i]}^{L}[\tau']\right)^q \ | \ t_i,t'_i \in \tau' \right] \bE'[ u(t'_{i},t_{i+1},\tau')^q \ | \  t'_i,t_{i+1} 
  \in \tau' ] \right) \\
  \times \tilde\bbE\left[\left(Z_{[t_{k},t'_k]}^{L}[\tau']\right)^q \ | \  t_k,t'_k \in \tau'  \right].
       \end{multline}
       Let us now observe that most of the terms in the above expression are translation invariant which will help for factorization.
       We have 
       \begin{multline}
         \bE' \tilde \bbE\left[\left(Z_{[t_{i},t'_i]}^{L}[\tau']\right)^q \ | \ t_i,t'_i \in \tau' \right]
         =  \bE' \tilde \bbE\left[\left(Z_{[0,t'_i-t_i]}^{L}[\tau']\right)^q \ | \ t'_i-t_i \in \tau' \right]\\
         \le  \tilde \bbE\left[\left( \bE'\left[Z_{[0,t'_i-t_i]}^{L}[\tau'] \ | \ t'_i-t_i \in \tau' \right] \right)^q \right],
                \end{multline}
                where we used Jensen's inequality in the last line.
          We can rewrite the last term as  
          \begin{equation}
          u(t'_i-t_i)^{-q}\tilde \bbE \left[ \left(\tilde Z^{L,\tilde \go}_{t'_i-t_i}\right)^{q}\right],
          \end{equation}
where  $\tilde Z^{L,\tilde \go}_{n}$ is the partition function associated to the (terminating) renewal $ \tilde \tau$ (with probability denoted by  $\tilde \bP^L$)
whose inter-arrival distribution is given by
\begin{multline}\label{deftildek}
  \tilde \bP^L[\tilde\tau_1=n]=\tilde K^L(n)\\
  :=\bP\otimes\bP'\left[ n\in\tau\cap \tau'\ ; \ (0,n)\cap\tau\cap \tau'=\emptyset \right] \ind_{\{n\le L\}}=:
  \tilde K(n)\ind_{\{n\le L\}}.
\end{multline}
We have
     \begin{equation}\label{oroko}
 \tilde Z^{L,\tilde \go}_{n}:= \tilde \bE^L \left[ \prod_{i=0}^n (1+ \gb \tilde \go_{i}\ind_{\{i \in \tilde \tau\}}) \ind_{\{n\in \tilde \tau\}}
 \right].
\end{equation}         
Regarding the contribution of the long jumps, we can ignore the constraint $\tau\cap \tau' \cap (a,b)=\emptyset$ in \eqref{defuz} and obtain that 
\begin{multline}
  \bE'[ u(t'_{i},t_{i+1},\tau')^q \ | \  t'_i,t_{i+1} 
  \in \tau' ]\\
  =\bE'[ u(0,t_{i+1}-t'_i,\tau')^q \ | \ t_{i+1}-t'_i\in \tau' ]\le u(t_{i+1}-t'_i)^q. 
\end{multline}
Factorizing $\bP'\left[\forall i \in \lint 1,k\rint,\ t_i,t'_i\in \tau' \right]$ and reorganizing the sum in the r.h.s.\ of \eqref{lasomcomplexe}, we have
       \begin{multline}\label{antes}
  \bE'\tilde \bbE \left[ (1 + \gb \tilde\go_{0})^q \left(Z_{N}^{\gb,\tilde \go}[\tau']\right)^q\right]\\
  \le \sum_{k\ge 1} \left(\sum_{m= L}^{\infty}  u(m)^{q+1} \right)^{k-1}\left(\sum_{n=0}^{\infty} u(n)^{1-q}
   \tilde \bbE \left[ \left(\tilde Z^{L,\tilde \go}_{n}\right)^{q}\right] \right)^k.
\end{multline}
Hence we can conclude the proof of Lemma \ref{conclouse} as soon as we can show that for some arbitrary $L$  
     \begin{equation}\label{plutiqueun}
\left(\sum_{m= L}^{\infty}  u(m)^{q+1} \right)\left(\sum_{n=0}^{\infty} u(n)^{1-q}
   \tilde \bbE \left[ \left(\tilde Z^{L,\tilde \go}_{n}\right)^q\right]\right) <1.
     \end{equation}
The first sum is easy to estimate considering Lemma \eqref{lem:don}. Using our assumption on $q$ \eqref{asumpq}, we have 
   \begin{equation}\label{loterm}
    \sum_{m= L}^{\infty} u(m)^{q+1}\le C L^{1-(1-\alpha)(q+1)},
   \end{equation}
   which can be made as small as we wish by choosing $L$ large.
  
\medskip

We need thus a uniform bound on the second sum. The idea is that for $\gb$ small, and ignoring the constraint of having no long jumps,
$\tilde Z^{L,\tilde \go}_{n}$ looks like the partition function of a pinning model in the 
delocalized phase, and thus it should be of order $P[n\in \tilde\tau]=u(n)^2$ (see \cite[Theorem 2.2]{cf:GB}).
Hence we should try to prove that  $  \tilde\bbE\left[ \left(\tilde Z^{L,\tilde \go}_{n}\right)^q \right]$ is of order $u(n)^{2q}$.
We prove the following result in the next section.
 
         \begin{lemma}\label{suZ}
     Given $q\le \gamma-1$,    
     there exist a constant $ C = C(q)$ and $\gb = \gb_0(L,q)$ such that for all $\gb\in (0,\gb_0)$ and 
        for all $n\ge 0$
     \begin{equation}\label{dabound}
          \tilde\bbE\left[\left(\tZ_{n}^{L,\tilde \go}\right)^{q}\right] \le  C u(n)^{2q}.           
     \end{equation}
\end{lemma}
        Recalling \eqref{asumpq} again, this result implies that
         \begin{equation}
       \sum_{n=0}^{\infty} u(n)^{1-q}
   \tilde \bbE \left[ \left(\tilde Z^{L,\tilde \go}_{n}\right)^q\right]\le  C \sum_{n \geq 0}
          u(n)^{1+q} \le C'.
         \end{equation} 
      Since the last constant $C'$ does not depend on $L$, we can combine this with \eqref{loterm} to prove \eqref{plutiqueun} for an adequate choice of $L$. 
      This concludes our proof of Lemma \ref{conclouse}.

       \subsection{Change of measure: the proof of Lemma \ref{suZ}} \label{changameas}

    The proof relies on the idea exposed at the beginning of the previous section.
    We use \eqref{davycrockett} for the partition function $\tZ_{[0,n]}^{L,\tilde \go}$ with 
      \begin{equation}
     G(\tilde\go):= \prod_{i=0}^n g(\tilde \go_i) \quad \text{where} \quad  g(\tilde \go_i):=(1+\gb\tilde\go_i)^{q-1}.
      \end{equation}
     Using  H\"older's inequality, we have
         \begin{equation}\label{cost}
          \tilde \bbE\left[\left(\tilde Z_{n}^{L,\tilde \go}\right)^{q}\right] \\ \leq 
          \left(\tilde \bbE\left[\prod_{i=1}^n g(\tilde \go_i)^{-\frac{q}{1-q}}\right] \right)^{1-q} \left(\tilde \bbE \left[\prod_{i=1}^n 
          g(\tilde\go_i)\tilde Z_{n}^{L,\tilde \go} \right] \right)^{q}.
            \end{equation} 
         Since the environment is IID, the first term is equal to
    \begin{equation}\label{cost2}
         \left(\tilde \bbE\left[g(\tilde \go_1)^{-\frac{q}{1-q}}\right]\right)^{(1-q)n},
     \end{equation}
 while the expectation in the second term is equal to (recall \eqref{deftildek})
\begin{equation}\label{benef}
 \sum_{k\ge 1} \sumtwo{0=t_0<\dots<t_k=n}{t_{i}-t_{i-1<L}} \tilde \bbE[g(\tilde\go_1)(1+\gb \tilde \go_1)]^{k+1} \prod_{i=1}^k \tilde K(t_i-t_{i-1}).
\end{equation}
Note that with our choice for $g$ and the definition \eqref{defitilt}, the terms integrated w.r.t. $\tilde \bbP$ appearing in \eqref{cost2} and \eqref{benef} are equal. Indeed 
   $$\tilde \bbE[g(\tilde \go_1)^{-\frac{q}{1-q}}]= \tilde \bbE\left[(1+\gb \tilde \go_1) g(\tilde \go_1)\right]=
   \bbE[ (1+\gb \go_1)^{1+q}].$$
Thanks to our assumption \eqref{asumpq}, this expectation is finite. 
 We also need it to be small and for this we will make use of the following immediate 
consequence of the Dominated Convergence Theorem.
         
             \begin{lemma}\label{chR}
             For any given $\gep > 0$, we can choose $\gb_{0}(\gep) > 0 $ such that for every $\gb \in (0,\gb_{0})$, we have 
              \begin{equation}\label{unifg}
             \bbE[ (1+\gb \go_1)^{1+q}] \leq (1+\gep).
                           \end{equation} 
              
           \end{lemma}
           In the following computations we choose $\gep=L^{-2}$.
           As a consequence of the above Lemma \ref{chR}, \eqref{cost} implies that for $\gb<\gb_0$ we have 
           \begin{multline}\label{cruize}
            \tilde \bbE\left[\left(\tilde Z_{n}^{L,\tilde \go}\right)^{q}\right]\le (1+\gep)^{(1-q)n}
            \left( \sum_{k\ge 1} \sumtwo{0=t_0<\dots<t_k=n}{t_{i}-t_{i-1<L}} (1+\gep)^{k+1}  \prod_{i=1}^k \tilde K(t_i-t_{i-1}) \right)^q\\
            \le  (1+\gep)^{(n+1)} \left( \sum_{k\ge 1} \sumtwo{0=t_0<\dots<t_k=n}{t_{i}-t_{i-1<L}}  \prod_{i=1}^k\tilde K(t_i-t_{i-1}) \right)^q.
           \end{multline}
Now  we fix 
$$c= -\log \left(\sum_{j=1}^{\infty} \tilde K(j)\right),$$ 
and we note that  
     \begin{multline}
 \sum_{k\ge 1} \sumtwo{0=t_0<\dots<t_k=n}{t_{i}-t_{i-1<L}} \prod_{i=1}^k \tilde K(t_i-t_{i-1}) \\
\le  \max \left( \sum_{k\ge n/L} \left(\sum_{j=1}^{\infty} \tilde K(j)\right)^k\ ,  \sum_{k\ge 1}\sum_{0=t_0<\dots<t_k=n}\prod_{i=1}^k \tilde K(t_i-t_{i-1}) \right).
\end{multline}
The second term in the $\max$ is equal to $\bP\otimes \bP' (n\in \tau \cap \tau')=u(n)^2$, 
and is a sufficient bound in the case where $n\le L^2$ as the pre-factor
 $(1+\gep)^{(n+1)}$ in \eqref{cruize} is bounded by $e$ in that case.
The first term is smaller than $(1-e^{-c})^{-1} e^{-cnL^{-1}}$ which together with \eqref{cruize} gives (cf. \eqref{doney})
\begin{equation}
  \tilde \bbE\left[\left(\tilde Z_{n}^{L,\tilde \go}\right)^{q}\right]\le (1-e^{-c})^{-1} (1+\gep)^{(n+1)} e^{-cnqL^{-1}}
  \le C u(n)^{2q}
  \end{equation}
for all $n\ge L^2$ provided that $L$ is chosen large enough.
 \qed

\section{Upper bound on the critical point shift}\label{upperboundcps}

In this section, we adapt the tools used in the proof of Theorem \ref{diso1} in order to obtain a lower bound on the free energy 
when $\alpha> 1- \gamma^{-1}$, which yields an upper bound on the critical point shift. More precisely we prove 
\begin{proposition}\label{dashift}
There exists a constant $C$ such that for every $\gb\in[0,1]$
\begin{equation}\label{dares}
h_c(\gb)\le \begin{cases} C \gb^{\gamma} \text{ if } \alpha\ge 1,\\
C |\log \gb| \gb^{\frac{\alpha\gamma}{1-\gamma(1-\alpha)}} \quad &\text{ if } \alpha\in(1/2,1),\\
                  C \gb^{\frac{\gamma}{2-\gamma}} |\log \gb|^{\frac{4-\gamma}{2(2-\gamma)}},  \quad &\text{ if } \alpha=1/2.
            \end{cases}
\end{equation}
If $\alpha\in (1-\gamma^{-1},1/2)$,
given $\delta>0$ there exists a constant $C_{\delta}>0$ such that for every $\gb\in[0,1]$
\begin{equation}
 h_c(\gb)\le C_\delta \gb^{\frac{\alpha\gamma}{1-\gamma(1-\alpha)}-\delta}.
 \end{equation}

\end{proposition}

\begin{rem}
While we tried to optimize the factor in front of $\gb$ power for $\alpha\ge 1/2$, we did not perform such an operation in the case
 $\alpha<1/2$, in order to keep the computation simple. In any case the correction needed would be worse than a logarithmic power.
 We believe that these corrections are artifacts of the proof, and that the asymptotic behavior of $h_c(\gb)$ should be given by a pure power of $\gb$ in most cases.
\end{rem}

The case $\alpha\ge 1$ is the easiest one: in that case the result follows directly from \eqref{hilde}.
Indeed it is a simple exercise to check that \eqref{defgamma} implies
\begin{equation}
 \bbE\left[ \log(1+\gb \go_n) \right] \stackrel{\gb\to 0}{\sim} - C_\bbP \gb^{\gamma} \left(\int_{0}^{\infty} \frac{x^{1-\gamma}}{1+x} \dd x\right). 
\end{equation}

 For the rest, we treat separately the two cases $\alpha \in (1- \gamma^{-1},1/2]$ and $\alpha\in (1/2,1)$, and use different methods for each.
 In both cases
we replace the environment $\go$ by a truncated version which ignores high values of $\go$.
In order to know where to perform the truncation we need to fix a referential size for the system.
For the rest of the computation we set 
\begin{equation}\label{defaN}
 N= N_{\gb}:= \begin{cases}
            c_1 \gb^{-\frac{\gamma}{1-\gamma(1-\alpha)}}, \quad &\text{ if } \alpha\in(1/2,1), \\
            c_1 \left(\gb^2|\log \gb| \right)^{-\frac{\gamma}{2-\gamma}},  \quad &\text{ if } \alpha=1/2,\\
             c_1 \gb^{-\frac{\gamma(1-\delta)}{1-\gamma(1-\alpha)}}, \quad &\text{ if } \alpha\in(1-\gamma^{-1},1/2),
       \end{cases}
\end{equation}
where $c_1$ is a small constant, and $\delta>0$ is chosen arbitrarily small.
Of course we have to consider the integer part of the above but we chose to omit it to simplify the notation.
Now we introduce $\hat \go^\gb$ given by 
\begin{equation}
 \hat \go^\gb_n:= \max(\go_n, N_\gb^{\gamma^{-1}}).
\end{equation}
The truncation level is chosen so that  $\hat \go^\gb_n$ is a reasonable approximation of $\go$ when restricted to the segment 
$[1,N_{\gb}]$; namely, with a positive probability (which does not depend on $\gb$) the two environments   coincide. We also set 
$$h_\gb:= -\log \bbE[1+\gb \hat \go^\gb].$$ 
The quantity $h_{\gb}$ is chosen in such a way that the partition function associated to the truncated environment has expected value one (for all values of $N$)
\begin{equation}
 \bbE[ Z^{\hat \go^\gb,\gb,\cf}_{N,h_\gb}]=1.
\end{equation}
Note that with our choice for $N_{\gb}$ we have 

\begin{equation}\label{hvalue}
h_\gb \sim \begin{cases} c'_{\alpha} \gb^{\frac{\alpha\gamma}{1-\gamma(1-\alpha)}}, \quad &\text{ if } \alpha\in (1/2,1), \\
            c'_{\alpha} \gb^{\frac{\gamma}{2-\gamma}} |\log \gb|^{\frac{\gamma-1}{2-\gamma}},  \quad &\text{ if } \alpha=1/2,\\
    c'_{\alpha} \gb^{\frac{\alpha\gamma}{1-\gamma(1-\alpha)}-\frac{\delta(\gamma-1)}{1-\gamma(1-\alpha)}}, \quad &\text{ if } \alpha\in(1-\gamma^{-1},1/2),
            \end{cases}
\end{equation}
where $c'_{\alpha}$ is a constant that depends on $c_1$ (in \eqref{defaN}) and $\alpha$.

The idea of our proof is to control either the second moment (if $\alpha\ge 1/2$) or the $p$-moment for some $p\in(1,2)$  (if $\alpha< 1/2$) 
of $Z^{\hat \go^{ \gb},\gb,\cf}_{N,h_{\gb}}$
in order to be able to apply the results of Section \ref{proofirel}.

\subsection{The case $\alpha \ge 1/2$}
 
The main thing we have to prove is that the second moment of the truncated partition function for the system of size $N_{\gb}$ is bounded. 
With this result at hand, we perform the same computation as in Section \ref{decompex} to conclude.

\begin{lemma}\label{secondmoment}
 If $c_1$ is chosen sufficiently small, we have
 \begin{equation}
 \sup_{\gb\in[0,1]} \bbE\left[ \left(Z^{\hat \go^{\gb},\gb,\cf}_{N_{\gb},h_\gb}\right)^2 \right]\le 2.
 \end{equation}
 \end{lemma}

As $\bbE\left[ \left(Z^{\hat \go^{\gb},\gb,\cf}_{N_{\gb},h_\gb}\right)^2 \right]<M$  for some constant $M$ that only depends on the choice $c_1$,
we can apply Lemma \ref{simplifiex} for $p=2$ 
and prove exactly as we proved Proposition \ref{irrelevex}, that there exist constants $C$ and $\delta$ such that for every $u\ge 0$
\begin{equation}
 \tf(\gb,h_\gb+u)\ge \frac{1}{N_{\gb}} \bbE \left[ \log  Z^{\hat \go^\gb,\gb,\cf}_{N_{\gb},h_\gb+u} \right]-C \frac{\log N_\gb}{N_{\gb}}\\
 \ge \delta N_\gb^{\alpha-1}u-C'\frac{\log N_{\gb}}{N_{\gb}}.
\end{equation}
Hence  setting 
$$u_\gb:=C'' N_{\gb}^{-\alpha} \log N_{\gb},$$ for a sufficiently large constant $C''$ we conclude that   
$$\tf(\gb, h_\gb+u_\gb )>0 \text{ and } h_c(\gb)\le h_\gb+u_\gb.$$
As for any value of $\alpha\in(1/2,1)$, $u_\gb$ is of a larger order magnitude than $h_{\gb}$,
we conclude that \eqref{dares} holds by replacing $N_{\gb}$ by its value.

\begin{proof}[Proof of Lemma \ref{secondmoment}]

Using the definition of $h_\gb$ we readily obtain that 
\begin{equation}\label{partint}
 \bbE\left[ \left(Z^{\hat \go^{\gb},\gb,\cf}_{N_{\gb},h_\gb}\right)^2 \right]=
 \bE^{\otimes 2}\left[ e^{\chi\sum_{n=1}^{N_{\gb}} \ind_{\{n\in \tau^{(1)}\cap \tau^{(2)}\}}} \right],
\end{equation}
where 
$$\chi=\chi(\gb):= \log(1+\gb^2 \bbE[(\hat \go^\gb_n)^2])$$  
and
$\tau^{(i)}$, $i=1,2$ are two IID renewal processes with distribution $\bP$.
Hence the quantity appearing in \eqref{partint} is simply the partition function of the homogeneous pinning model associated to $\tilde \tau$. Note that
since $\alpha \ge  1/2$, the renewal process $\tilde \tau:=\tau^{(1)}\cap \tau^{(2)}$ is recurrent.

\medskip

\noindent Repeating the computations performed in \cite[Equations (6.24)-(6.27)]{cf:BL}, we obtain 
\begin{equation}
 \bbE\left[ \left(Z^{\hat \go^{\gb},\gb,\cf}_{N_{\gb},h_\gb}\right)^2 \right]
 \le 1+\chi \sum_{k=1}^{N_\gb}\exp\left(k \left[\chi+ \log\bP^{\otimes 2}[\tilde \tau_1\le N_{\gb}]\right]\right).
\end{equation}
This is uniformly bounded by $2$ if we can show that 
\begin{equation}\label{copachi}
\log\bP^{\otimes 2}[\tilde \tau_1\le N_{\gb}]\le -2\chi.
\end{equation}
We observe that the function
\begin{equation}
 D(N):=\sum_{n=1}^N \bP^{\otimes 2}[n \in \tilde \tau]=\sum_{n=1}^N u(n)^2
\end{equation}
is regularly varying: according to \eqref{doney} it asymptotically equivalent to either $\log N$ (if $\alpha=1/2$) or $N^{2\alpha-1}$ times a constant (if $\alpha\in (1/2,1)$).
Thus we can use \cite[Theorem 8.7.3]{cf:BGT} which implies that 
\begin{equation}
 \bP^{\otimes 2}[\tilde \tau_1> N] \stackrel{n\to \infty}{\sim} \frac{1}{\gG(2\alpha)\gG(2(1-\alpha)) D(N)}.
\end{equation}
This yields that for an appropriate constant which we do not wish to compute that 
\begin{equation}
-\log\bP^{\otimes 2}[\tilde \tau_1\le N_{\gb}]\stackrel{\gb\to 0}{\sim} \begin{cases}
                                                                         C(K) (\log N_\gb)^{-1}, \ga = 1/2\\
                                                                         C(K) N_\gb^{1-2\alpha}, \ga \in (1/2,1).
                                                                        \end{cases}
\end{equation}
On the other hand from \eqref{defgamma} we have 
\begin{equation}
\chi=\chi(\gb):= \log(1+\gb^2 \bbE[(\hat \go^\gb_n)^2])\stackrel{\gb\to 0}{\sim} \frac{C_{\bbP}}{2-\gamma} \gb^2 N_{\gb}^{2-\gamma}.
\end{equation}
Replacing $N_\gb$ by its value, we can check that \eqref{copachi} is satisfied for $\gb$ sufficiently small if $c_1$ in \eqref{defaN}
is chosen small enough.
 
\end{proof}

\begin{rem}
 When $\alpha<1/2$, the method exposed above would also provide an upper bound on $h_c(\gb)$.
 As in that case the intersection of two independent renewals is a terminating renewal, one would have to chose $N_{\gb}$ in a way such that 
 $\chi(\gb)$ remains bounded by a small constant.
 However, this would not yield  the right exponent for $h_c(\gb)$.
 \end{rem}

\subsection{The case $\alpha\in(1-\gamma^{-1},1/2)$}

In this other case we do not require the size of the system to be $N_{\gb}$.
The important statement to prove is the following.

\begin{lemma}\label{pmoment}
 If $\gb\le \gb_0$ and $c_1$ is chosen sufficiently small, there exists $p>1$ such that 
 \begin{equation}
\sup_{N\ge 0} \bbE\left[ \left(Z^{\hat \go^{\gb},\gb,\cf}_{N,h_\gb}\right)^{p} \right]<\infty.
 \end{equation}
 \end{lemma}
 
 Following the steps of Section \ref{decompex} we can deduce from Lemma \ref{pmoment} that for any $u>0$ we have 
 \begin{equation*}
  \tf(\gb,h_{\gb}+u)\ge   \frac{1}{N} \bbE \left[ \log  Z^{\hat \go^\gb,\gb,\cf}_{N_{\gb},h_\gb+u} \right]-C \frac{\log N}{N}
  \ge  \delta N_\gb^{\alpha-1}u-C'\frac{\log N_{\gb}}{N_{\gb}}.
 \end{equation*}
 Choosing $N$ sufficient large (depending on $u$), the r.h.s.\ becomes positive, which implies $h_c(\gb)\ge h_{\gb}$.
This completes the proof of Proposition \ref{dashift}.

 \begin{proof}[Proof of Lemma \ref{pmoment}]
 
 We choose 
 $$p=\frac{1+\delta'}{1-\alpha}<2,$$
  where $\delta'>0$ is to depend on $\delta$ (which enters in the definition of $h_{\gb}$) in a way which we  determine later.
  In view of \eqref{antes}, it is sufficient to show that for some $L\in \bbN$
\begin{equation}\label{eq:goooz}
  \left(\sum_{m\ge L} u(m)^{p} \right)\left( \sum_{n=0}^{\infty} u(n)^{2-p} 
   \tilde \bbE^{\gb}\left[ \left(\tilde Z^{L,\tilde \go^{\gb}}_{n,h_{\gb}} \right)^{p-1} \right]\right)<1,
\end{equation}
where $u(m)$ is defined in \eqref{doney} and
 the tilted measure $\tilde \bbP^{\gb}$ for the 
environment $\tilde \go^{\gb}$ is defined 
as a product measure on $\bbZ$ whose marginal law is 
\begin{equation}
 \tilde \bE^{\gb}[ \tilde \go^{\gb}\in \dd x]=e^{h_{\gb}}(1+\gb x)\bbP[\hat \go^{\gb}\in \dd x],
\end{equation}
and the partition function $\tilde Z^{L,\tilde \go^{\gb}}_{n,h_{\gb}}$ is similar to the one defined as in \eqref{oroko}, but including $h_{\gb}$. More 
explicitly its definition is the following 
\begin{equation}
\tilde Z^{L,\tilde \go^{\gb}}_{n,h_{\gb}}=\tilde \bE^L\left[ \prod_{i=0}^n \left( e^{h_{\gb}}\gb \tilde \go^{\gb}_i \ind_{\{i \in \tilde \tau^L\}} + \ind_{\{i \notin \tilde \tau^L\}} \right)\right].
\end{equation}
In that case, one deduces from Lemma \ref{lem:don} that the first sum \eqref{eq:goooz} is smaller than
\begin{equation}
  \sum_{m\ge L} u(m)^{p}\le C L^{-\delta'}
\end{equation}
which tends to zero when $L$ tends to infinity.
To control the second term we must prove that 
\begin{equation}
  \tilde \bbE^{\gb}\left[ \left(\tilde Z^{L,\tilde \go^{\gb}}_{n,h_{\gb}} \right)^{p-1} \right]\le C u(n)^{2(p-1)}.
\end{equation}
The key point  is to show that given $L$, $p$ and $\gep > 0$, for $\gb$ sufficiently small, we have
\begin{equation}\label{chR2}
 \bbE\left[ \left[e^{h_{\gb}}(1+\gb\hat \go^\gb_1)\right]^p\right]\le 1+\gep,
\end{equation}
as the rest follows exactly as in the proof of Lemma \ref{chR}.
First we notice that the $e^{h_{\gb}}$ term can be neglected as it tends to one.
As for the rest, as $p/2 < 1$ we have 
\begin{equation}
(1+\gb\hat \go^\gb_1)^p=(1+2\gb\hat \go^\gb_1+(\gb\hat \go^\gb_1)^2)^{p/2} \le 1+ (2 \gb |\hat\go^\gb_1|)^{p/2}+ \gb|\hat \go^\gb_1|^p
\end{equation}
The expectation of the second term is smaller than $(2\gb \bbE[|\hat \go_1|])^{p/2}$ and thus tends to zero when $\gb$ tends to $0$. 
As for the last term, using \eqref{defgamma} and replacing $N_\gb$ by its value, we obtain
\begin{equation}
\bbE\left[ | \gb\hat \go^\gb_1|^p\right]\le C\gb^p N_{\gb}^{\frac{p-\gamma}{\gamma}} =C \gb^{\frac{\delta}{1-\alpha}-\frac{\delta'\gamma}{1-\gamma(1-\alpha)}+\frac{\delta\delta'}{(1-\alpha)(1-\gamma(1-\alpha))}}.
\end{equation}
The above exponent is positive if $\delta'(1-\alpha)\gamma\le [1-\gamma(1-\alpha)]\delta$. 
This allows to conclude that \eqref{chR2} holds if $\gb$ is sufficiently small.

 \end{proof}

\section{Disorder relevance}\label{rerel}

\subsection{Reduction to a fractional moment bound}

The aim of this section is to prove that when $\alpha > 1-\gga^{-1}$, there is a critical point shift whose magnitude is of order 
$\gb^{\frac{\ga \gga}{(\ga-1)\gga +1}}$ up to $\log$ correction.
To prove such a statement we adapt the technique used in \cite{cf:DGLT07} which combines the use of a finite volume criterion 
(Proposition \ref{prop:derau}) and a change of measure argument.

\medskip

Although some refinements of this method involving a more delicate coarse graining procedure have been developed in the literature (introduced in 
\cite{cf:T09} and developed in  e.g.\ \cite{cf:GLT08, cf:BL})
and should yield a better estimate on  $h_c(\gb)$, a lighter proof seemed more appropriate here.
Let us set 

\begin{equation}\label{defhbet}
 h^{(2)}_{\gb}= \begin{cases} c_2  \left( \frac{\gb}{|\log \gb|+1} \right)^{\frac{\alpha\gamma}{1-\gamma(1-\alpha)}} \quad  &\text{ if }  \alpha \in (1-\gamma^{-1},1),\\
  c_2 \frac{\gb^{\gamma}}{(|\log \gb|+1)^{3\gamma-2}}, \quad &\text{ if } \alpha=1,\\
  c_2  \left( \frac{\gb}{(|\log \gb|+1)}\right)^{\gamma}, \quad &\text{ if } \alpha>1.
         \end{cases}
\end{equation}

The conclusion we obtain is the following.

\begin{proposition}\label{propint}
         If $\alpha > 1-\gamma^{-1}$, then  $h_c(\gb)>0$ for every $\gb \in(0,1)$.
         Moreover there exists a choice for the constant $c_2$ in \eqref{defhbet} such that for all $\gb$
       \begin{equation} \label{disrel}
h_c(\gb) \geq  h^{(2)}_{\gb}.
\end{equation} 
\end{proposition}

 A first important observation is that in order to control the free energy, it is sufficient to control the growth of fractional moments of the partition function.
Indeed, we have, for any $\theta\in (0,1)$
\begin{equation}
 \bbE \left[ \log Z^{\gb,\go}_{N,h} \right]= \frac{1}{\theta}\bbE \left[ \log (Z^{\gb,\go}_{N,h})^{\theta} \right]  \le
 \frac{1}{\theta} \log  \bbE \left[ (Z^{\gb,\go}_{N,h})^{\theta} \right].
\end{equation}
Hence to show that $\tf(\gb,h^{(2)}_\gb)=0$, we only need to prove that  $A_{N} := \E[(Z_{N,h^{(2)}_\gb}^{\gb,\go})^{\theta}]$ is uniformly bounded (we set by convention $A_0=1$).
We will do so by using a bootstrapping  argument from \cite{cf:DGLT07} which shows that if the first $k$ values $(A_{1},\ldots,A_{k})$ are small enough then
$A_N$ is uniformly bounded. More precisely,  given $\gb\in (0,1)$, $h\in \bbR$, $k\in \bbN$ and $\theta\in(0,1)$, we set 
\begin{equation}\label{def:rho}
 \rho(\gb,h,k,\theta) := \bE\left[e^{\theta h}(1+\gb\go_1)^{\theta}\right] \sum_{n=k}^{\infty} \sum_{j=0}^{k-1} K(n-j)^{\theta} A_{j} 
 \end{equation}
 We have the following:      
        \begin{proposition}[Proposition 2.5 in \cite{cf:DGLT07}]\label{prop:derau}
  Given $h$ and $\gb \in (0,1)$, if we can find $k\in \bbN$ and $\theta\in (0,1)$ such that 
  $$\rho (\gb,h,k,\theta)\le 1,$$
 then $\tf(\gb,h) =0$. 
        \end{proposition}
The proof is very short and relies on a decomposition of the partition function $Z^{\gb,\go}_{N,h}$ 
according to the position of the first contact point in the interval
$[N-k+2,N]$ and the last one in the interval $[0,N]$, and the use of \eqref{eq:subadt}.

\medskip

\noindent For the rest of this section, we fix $h=h^{(2)}_{\gb}$,
\begin{equation}\label{defkk}
  k=\begin{cases}  h^{-\frac{1}{\alpha}}, \quad &\text{ if } \alpha \in (1-\gamma^{-1},1),\\
   h^{-1} |\log h|^{-2}\quad &\text{ if } \alpha=1,\\
   h^{-1}  \quad &\text{ if } \alpha>1.\\
  \end{cases}
 \end{equation}
 and choose $\theta=1-(\log k)^{-1}$.
 
 \medskip
 
In this setup, to prove that $\rho$ is smaller than one, it is enough to show that for all $j<k$ $A_{j}$ 
is significantly smaller (that is by a large constant factor) than $u(j):=P[j\in \tau]$.
More precisely we need the following estimate

\begin{lemma}\label{fracmoment}
 There exists $\eta(c_2)>0$ which can be made arbitrarily small by choosing $c_2$ adequately such that
 for all $j\ge \eta k$
 \begin{equation}
  A_j\le  \eta u(j).
 \end{equation}
Moreover for $c_2$ sufficiently small we also have for all $j\le \eta k$
\begin{equation}
 A_j\le 2e^2 u(j).  
\end{equation}
\end{lemma}

 \begin{proof}[Proof of Proposition \ref{prop:derau}]
  Let us focus first on the case $\alpha\in (1-\gamma^{-1},1)$ for simplicity
  By Jensen's inequality
  \begin{equation}
  \bE\left[(e^{h}(\gb \go_{1} +1))^{\theta}\right]\le e^{\theta h} \le 2,
  \end{equation}
provided constants are chosen adequately.
  Now using assumption \eqref{Kt}, the definition of $\theta$ and usual comparisons between sums and integrals, we get
  \begin{multline}
  \sum_{n=k}^{\infty} K(n-j)^{\theta}\le C \int_{k-j}^{\infty} x^{-(1+\alpha)\theta}\dd x 
  = \frac{C}{(1+\alpha)\theta-1} (k-j)^{1-(1+\alpha)\theta}\\
  \le C'\exp((1+\alpha)(1-\theta) \log (k-j)) (k-j)^{-\alpha} \le C''(k-j)^{-\alpha}.
  \end{multline}
  Thus, using Lemma \ref{fracmoment} and \eqref{doney}, we have
  \begin{equation}
 \rho \le C \sum_{j=0}^{k-1}  (k-j)^{-\alpha} A_i
 \le C  \eta \sum_{j=\eta k}^{k-1}  (k-j)^{-\alpha} j^{\alpha-1}+ C'  \sum_{j= 0}^{\eta k} j^{\alpha-1} (k-j)^{-\alpha}.
  \end{equation}
where in the last line we have used \eqref{doney}. The reader can check that the sum can be made arbitrarily small by choosing $\eta$ small.

\medskip

The proof in the case $\alpha\ge 1$ follows exactly the same line, we leave the details of the computation to the reader.

\end{proof}

\subsection{Proof of Lemma \ref{fracmoment}}

The idea is to use a change of measure similar to the one used in the proof  of Lemma \ref{chR}, that is  which penalizes
environments which make the partition function very large. 
since these might give a significant contribution to $\bbE[Z^{\gb,\go}_{i,h}]$, and not to $A_i$.
Given $G$ such a penalization function we have by H\"older's inequality

\begin{equation}\label{eq:holdit}
 A_j \le \bbE[G(\go)^{-\frac{\theta}{1-\theta}}]^{1-\theta} \left(\bbE[G(\go)Z^{\gb,\go}_{j,h}]\right)^{\theta}.
\end{equation}
We choose to use a change of measure that penalizes high values for $\go$. 
As the exponent $\frac{\theta}{1-\theta}$ is large, we cannot choose a very high penalization. Thus we set
\begin{equation}
 G(\go):= \exp\left(- (\log k)^{-1}\sum_{i=1}^k \ind_{\{\go_i\ge k^{\gamma^{-1}} \}}\right)=:\prod_{i=1}^k g(\go_i).
\end{equation}
We are going to prove the two following results

\begin{lemma}\label{lem:cout}
There exists a constant $C$ such that for any value of $k$,
\begin{equation}
  \bbE[G(\go)^{-\frac{\theta}{1-\theta}}]\le C.
 \end{equation}
\end{lemma}

\begin{lemma}\label{lem:benef}
There exists $\eta(c_2)$, which can be chosen arbitrarily small if $c_2$ is chosen adequately, such that 
\begin{equation}
 \bbE[G(\go)Z^{\gb,\go}_{j,h}]\le \left(\ind_{\{j< \eta k\}}+\eta\ind_{\{ j\ge \eta k \}} \right)  u(j).                               
\end{equation}
\end{lemma}

Lemma \ref{fracmoment} can be immediately deduced from \eqref{eq:holdit}, Lemma \ref{lem:cout} and Lemma \ref{lem:benef}.
Indeed from Lemma \ref{lem:cout}  the first term in \eqref{eq:holdit} can be bounded by $2$ if $k$ is sufficiently large (that is $c_2$ sufficiently small),
while the second term is smaller than 
$$\left(\ind_{\{j< \eta k\}}+\eta^{\theta}\ind_{\{ j\ge \eta k \}} \right)  u(j)^{\theta},$$
and we can conclude using the fact that if $k$ is sufficiently large, from \eqref{doney} and our choice for $\theta$, for all $j\le k$
$$u(j)^{\theta} \ge e^{2}u(j).$$

\begin{proof}[Proof of Lemma \ref{lem:cout}]
 Because of the assumption on the tail distribution of $\go$ we have 
 $$\bbP\left[ \go_j\ge k^{\gamma^{-1}} \right]\le C k^{-1}.$$
Hence
 \begin{equation}
  \bbE[G(\go)^{-\frac{\theta}{1-\theta}}]\le  \bbE\left[e^{\sum_{j=1}^k \ind_{\{\go_j\ge k^{\gamma^{-1}} \}}} \right]
  =\left( 1+(e-1) \bbP\left[ \go_j\ge k^{\gamma^{-1}} \right] \right)^k\le C'.
 \end{equation}
\end{proof}

\begin{proof}[Proof of Lemma \ref{lem:benef}]
Using the product structure of $G$, we have  
\begin{equation}\label{groosa}
  \bbE[G(\go)Z^{\gb,\go}_{j,h}]:= \bE \left[ \exp\left(-\sum_{i=1}^k\left(\eta_1 \ind_{\{n\notin\tau \cap[1,j]\}})+\eta_2 \ind_{\{i\in\tau\cap[1,j]\}} \right) \right)\right],
 \end{equation}
where 
\begin{equation}
 \eta_1:= -\log \bbE\left[g(\go_1)\right] \text{ and }  \eta_2:= -\log \bbE\left[g(\go_1)(1+\gb \go_1)\right]-h.
\end{equation}
Both $\eta_1$ and $\eta_2$ are positive for our choices of parameter (and this fact is sufficient to treat the case $j\le \eta k$). 
For $\eta_1$ this is obvious as $g(\go_1)\le 1$. 
Concerning $\eta_2$ we note that
\begin{equation}
  \bbE\left[g(\go_1)(1+\gb \go_1)\right]\le 1+ \gb \bbE[ g(\go_1)\go_1 ]
\end{equation}
and as $\bbE[\go_1]=0$, we have 
 \begin{multline}
  \bbE[ g(\go_1)\go_1 ]=   \bbE[ (g(\go_1)-1)\go_1 ] \\
  = \left(e^{-(\log k)^{-1}}-1\right)\bbE\left[\go_1 \ind_{\{\go_1\ge k^{\gamma^{-1}}\}}\right]\le 
  -c (\log k)^{-1} k^{\frac{1-\gamma}{\gamma}},
 \end{multline}
where $c$ is a positive constant that only depends on the distribution of $\go_1$.
Thus we have
\begin{equation}\label{titadouz}
 \eta_2\le c\gb (\log k)^{-1} k^{\frac{1-\gamma}{\gamma}}-h.
\end{equation}
Using our choice of parameters (recall \eqref{defkk}), we see that when $c_2$ is chosen sufficiently small, 
the r.h.s.\ in \eqref{titadouz} is negative and that $\eta_2\le -\tilde h$
\begin{equation}
 \tilde h(M,k)=\begin{cases} M k^{-\min(\alpha,1)} \quad & \text{ if } \alpha\ne 1,\\
                M k^{-1} (\log k)^{2} \quad & \text{ if } \alpha=1,
               \end{cases}
\end{equation}
where $M$ can be chosen arbitrarily large if $c_2$ is chosen sufficiently small.
Hence recalling \eqref{groosa}, we have 
\begin{equation}\label{groosa2}
  \bbE[G(\go)Z^{\gb,\go}_{j,h}]\le \bE\left[ e^{-\tilde h \sum_{i=1}^j \delta_i} \delta_j\right]
  = u(j)  \bE\left[ e^{-\tilde h \sum_{i=1}^j \delta_i} \ | \ j\in \tau\right].
\end{equation}
To conclude, we need to show that the r.h.s.\ can be made much smaller than $u(j)$ for all $j\ge \eta k$ if $M$ is chosen sufficiently large.
To this purpose it is sufficient to use the following result (proved below).
\begin{lemma}\label{nocontax}
 We set
 \begin{equation}
\cA(N,\gep)= \begin{cases} \left\{\cN_N(\tau) \le \gep N^{\min(\alpha,1)}\right\} \quad & \text{ if } \alpha \ne 1,\\
              \left\{\cN_N(\tau) \le \gep N (\log N)^{-2}\right\} \quad & \text{ if } \alpha=1.
             \end{cases}
\end{equation}
We have 
\begin{equation} 
 \lim_{\gep \to 0} \limsup_{N\to \infty} \bP\left[  \cA(N,\gep) \ | \ N\in \tau \right]=0,
\end{equation}
\end{lemma}
and for $j\ge \eta k$,
 \begin{equation}
  \bE\left[ e^{-\tilde h \sum_{i=1}^j \delta_i} \ | \ j\in \tau\right]
  \le \bP\left[ \cA(j,\gep) \ | \ j\in \tau \right]+ e^{-\gep j^{\min(\alpha,1)}\tilde h }.
 \end{equation}
If we choose $\gep$ sufficiently small and $j$ sufficiently large, both terms can be made arbitrarily small (and the same argument works for $\alpha=1$).

\end{proof}

\begin{proof}[Proof of Lemma \ref{nocontax}]
 
 In the case $\alpha>1$, the result is a consequence of the law of large numbers, since the renewal theorem (recall \eqref{doney}) yields that  $N\in \tau$ has a probability
 bounded away from zero.
 The case $\alpha<1$ is a direct consequence \cite[Proposition A.2]{cf:DGLT07}. For the case $\alpha=1$, we note that 
 \begin{multline}\label{laprobab}
   \bP\left[  \cA(N,\gep) \ ; \ N\in \tau \right]= \bP\left[  \sum_{i=1}^{\gep N (\log N)^{-2}} (\tau_{i}-\tau_{i-1}) =  N\right] \\
   =   \bP\left[  \sum_{i=1}^{\gep N (\log N)^{-2}} (\tau_{i}-\tau_{i-1})\ind_{\{\tau_{i}-\tau_{i-1}\}\le N\}} =  N\right].
 \end{multline}
Now we notice that 
\begin{equation}
\bE\left[  \sum_{i=1}^{\gep N (\log N)^{-2}} (\tau_{i}-\tau_{i-1})\ind_{\{\tau_{i}-\tau_{i-1}\}\le N\}} \right]=\gep N (\log N)^{-2} \bE\left[ \tau_1 \ind_{\{\tau_1\le N\}}\right]
\le C \gep N\log N^{-1}.
\end{equation}
Hence the probability in \eqref{laprobab} is smaller than $C\gep \log N^{-1}$ from Markov's inequality, and we can conclude using \eqref{doney}.
\end{proof}

\bigskip
 
 {\bf Acknowledgements:} This work was initiated during a visit of J.S.\ at IMPA, he acknowledges hospitality and support.
 H.L. acknowledges support  of a productivity grant from CNPq.

\end{document}